\documentclass {amsart}
\usepackage{verbatim}
\usepackage{amssymb,amsmath}
\usepackage{amsthm}
\usepackage{graphicx}
\usepackage{epsfig}
  
\newtheorem{theorem}{Theorem}
\newtheorem{corollary}{Corollary}[section]
\newtheorem{lemma}[corollary]{Lemma}
\newtheorem{proposition}[corollary]{Proposition}

\newcommand{\Prob} {{\mathbb P}}

\newcommand{\R}{{\mathbb{R}}}
\newcommand{\C}{{\mathbb C}}

\newcommand{\dist}{{\rm dist}}

\def \Re {{\rm Re}}
\def \p {\partial}

\def \Half {{\mathbb H}}

\def \Disk {{\mathbb D}}

\def \diam {{\rm diam}}

\def \F  {{\mathcal F}}
\def \G  {{\mathcal G}}

\def \rect {{\mathcal R}}

\def \arcs {{\mathcal A}}

\def \Heuristic {\noindent {$\clubsuit$}}

\def \linehere { {\hrule}}

\def \mtwo {{\medskip \medskip}}
\def \ms {{\medskip}}

\def \labove { \mtwo \linehere \ms   }

\def \lbelow {{\ms \linehere \mtwo}}

\newenvironment{definition}[1][Definition]{\begin{trivlist}
\item[\hskip \labelsep {\bfseries #1}]}{\end{trivlist}}

\begin{document}

\title[Continuity of radial and two-sided radial $SLE$]
{Continuity of radial and two-sided radial $SLE_\kappa$ at the terminal
point}

\author{Gregory F. Lawler}
\address{Department of Mathematics\\
University of Chicago }
\email{lawler@math.uchicago.edu}
\thanks{Research supported by National
 Science Foundation grant DMS-0907143.}

 
\maketitle

\begin{abstract}  
We prove that  radial  $SLE_\kappa$  and two-sided radial
$SLE_\kappa$ are   
continuous at their terminal point.
\end{abstract}

\section{Introduction}

We  answer a question posed  
by Dapeng Zhan about  radial Schramm-Loewner
evolution ($SLE_\kappa$) and discuss a similar
question about two-sided $SLE_\kappa$ that arose
in work of the author with Brent Werness \cite{LBrent}.
  Radial $SLE_\kappa$ was invented by
 Oded Schramm \cite{Schramm} and 
is a one-parameter family of random curves 
\[ \gamma: [0,\infty) \rightarrow \overline \Disk, \;\;\;\;
  \gamma(0) \in \p \Disk, \] 
where $\Disk$ denotes the unit disk.
  The definition implies that
$\gamma(t) \neq 0$ for every $t$ and
\[          \liminf_{t \rightarrow \infty} |\gamma(t)|
    = 0 . \]
Zhan asked for a proof that with probability one
\begin{equation}  \label{may8.1}
       \lim_{t \rightarrow \infty} \gamma(t) = 0 . 
\end{equation}
For $\kappa > 4$, for which the $SLE$ paths intersect themselves,
this is not difficult to prove because the path makes closed
loops about the origin.  The harder case is $\kappa \leq 4$.
Here we establish \eqref{may8.1} for $\kappa \leq 4$ by proving
a stronger result.

To state the result, let
\[  \Disk_n = e^{-n} \Disk = \left \{z \in \C: |z| < e^{-n} \right\}, \]
\[      \rho_n = \inf\left\{t: |\gamma(t)| = e^{-n} 
 \right\}, \]
and let $\G_n$ denote the $\sigma$-algebra generated by
$\{\gamma(s): 0 \leq s \leq \rho_n\}$.  We fix
\[  \alpha =  \frac 8 \kappa -1 ,\]
which is positive for $\kappa < 8$.

\begin{theorem} \label{radmain}
 For every $0 < \kappa < 8$, there
exists $c > 0$ such that if $\gamma$ is radial $SLE_\kappa$ from
$1$ to $0$ in $\Disk$ and $j,k,n$ are positive integers,
then
\begin{equation}  \label{apr4.100}
  \Prob\left\{\gamma[\rho_{n+k},\infty) \subset \Disk_j
 \mid \G_{n+k}\right\} \geq [1- c \, e^{-n\alpha/2}] \, 
  \, 1\{\gamma[\rho_k,\rho_{n+k}] \subset \Disk_j \}. 
  \end{equation}
Moreover,   if $0 < \kappa \leq 4$, then
\begin{equation}  \label{jul20.1} 
 \Prob\left\{\gamma[\rho_{n+k},\infty) \subset \Disk_k
 \mid \G_{n+k}\right\} \geq 1- c \, e^{-n\alpha/2}. 
 \end{equation}
 \end{theorem}


There is another version of $SLE$, sometimes called
two-sided radial $SLE_\kappa$ which corresponds to
chordal $SLE_\kappa$ conditioned to go through an
interior point.  We consider the case of chordal $SLE_\kappa$
in $\Disk$ from $1$ to $-1$ conditioned to go
through the origin stopped when it reaches
the origin (see Section \ref{twosec}
 for precise definitions).

\begin{theorem} \label{twomain}
 For every $0 < \kappa < 8$, there
exists $c > 0$ such that if $\gamma$ is two-sided
radial $SLE_\kappa$ from
$1$ to $-1$ through
$0$ in $\Disk$ and $j,k,n$ are positive integers,
then
\begin{equation}  \label{apr4.1}
   \Prob\left\{\gamma[\rho_{n+k},\infty) \subset \Disk_j
 \mid \G_{n+k}\right\} \geq [1- c \, e^{-n\alpha/2}] \, 
  \, 1\{\gamma[\rho_k,\rho_{n+k}] \subset \Disk_j \}. 
  \end{equation}
\end{theorem}
 
Using these theorem, we are able to obtain the following
corollary.  Unfortunately, we are not able to estimate
the exponent $u$ that appears.

\begin{theorem} \label{maincorollary}
 For every $0 < \kappa < 8$, there exist
 $c < \infty, u > 0$ such that the following
 holds.  Suppose
$\gamma$ is either radial $SLE_\kappa$ from $1$ to $0$
in $\Disk$ or two-sided radial $SLE_\kappa$ from $1$ to $-1$
through $0$ stopped when it reaches the origin. 
Then, for all nonnegative integers $k,n$, 
\begin{equation}  \label{apr6.1}
   \Prob\{\gamma[\rho_{n+k},\infty) \cap \p \Disk_k
\neq \emptyset \mid \G_k \} \leq c \, e^{-un}, 
\end{equation}
and hence
\begin{equation}  \label{apr5.1}
 \Prob\{\gamma[\rho_{n+k},\infty) \cap \p \Disk_k
\neq \emptyset \} \leq c \, e^{-un}. 
\end{equation}
In particular, if $\gamma$ has the radial parametrization, then
with probability one,
\[         \lim_{t \rightarrow \infty} \gamma(t) = 0 . \]
\end{theorem}

Note that  \eqref{apr6.1}  is not as strong a result as \eqref{jul20.1}.  At
the moment, we do not have uniform bounds for
\[  \Prob\{\gamma[\rho_{n+k},\infty) \cap \p \Disk_k
\neq \emptyset \mid \G_{n+k} \} \]
for radial $SLE_\kappa$ with $4 < \kappa < 8$ or two-sided radial
$SLE_\kappa$ for $0 < \kappa < 8$.

\labove \textsf
{\Heuristic  \begin{small}  
There is another, perhaps easier, way of obtaining \eqref{apr6.1}
for radial $SLE_\kappa$, $4 < \kappa < 8$, by using the fact
that the curve hits itself (and hence also forms closed loops
about origin).  This approach, however, does not work for $\kappa \leq 4$
or for two-sided radial for $\kappa < 8$ since in these cases
the origin is not separated from $\p \Disk$ in finite time.
 \end{small}}
\lbelow

\subsection{Outline of the paper}

When studying $SLE$, one uses many kinds of estimates: results
for all conformal maps; results that hold for solutions of
the (deterministic) Loewner differential
equation; results about
stochastic differential equations (SDE), often simple equations
of one variable; and finally results that combine them all.
We have
separated the non-$SLE$ results into
a  ``preliminary'' section with subsections emphasizing the
different aspects. 

We discuss three kinds of $SLE_\kappa$: radial, chordal,
and two-sided radial.  They are probability measures on curves 
(modulo reparametrization) in simply connected
domains   connecting, respectively:
boundary point to interior point, two distinct boundary points, and
two distinct boundary points conditioned to go through an interior
point.  In all three cases, the measures are conformally invariant
and hence we can choose any convenient domain.  For the radial
equation, the unit disk $\Disk$ is most convenient and for this
one gets the Loewner equation as originally studied by Loewner. For this
equation a {\em radial parametrization} is used which depends on the
interior point.  For the
chordal case, Schramm \cite{Schramm} showed that the half-plane
with boundary points $0$ and $\infty$ was most convenient, and the
corresponding Loewner equation is probably the easiest for studying
fine properties. Here a {\em chordal parametrization} depending
on the target boundary point (infinity) is most convenient.
  The two-sided radial, which was 
introduced in \cite{LLind,LPark} and can be considered as a type of
$SLE(\kappa,\rho)$ process as defined in \cite{LSWrest}, 
has both an interior point and
a boundary point.  If one is studying this path up to the time
it reaches the interior point, which is all that we do in this paper, 
then one can use either the radial or the chordal parametrization.

The three kinds of $SLE_\kappa$, considered as measures on curves
modulo reparametrization,  are locally absolutely continuous
with respect to each other.  To make this precise, it is
easiest if one studies them simultaneously in a single domain with
a single choice of parametrization.  We do this here choosing the
radial parametrization in the unit disk $\Disk$.  We review the radial
Loewner equation in Section \ref{radsec}.  We write the equation slightly
differently than in \cite{Schramm}.  First, we add a parameter
$a$ that gives a linear time change.  We also write a point on the
unit circle as $e^{2i\theta}$ rather than $e^{i\theta}$; this makes 
the SDEs slightly easier and also shows the relationship between this
quantity and the argument of a point in the chordal case.  Indeed,
if $F$ is a conformal transformation of the unit disk to the
upper half plane with $F(1) = 0$ and $F(e^{2i\theta}) = \infty$, then
$\sin[\arg F(0)] = \sin \theta$.

The radial Loewner equation describes the evolution of a curve
$\gamma$ from $1$ to $0$ in $\Disk$. More precisely, if $D_t$
denotes the connected component of $\Disk \setminus \gamma(0,t]$
containing the origin, and $g_t: D_t \rightarrow \Disk$ is
the conformal transformation with $g_t(0) = 0, g_t'(0) > 0,$
then the equation describes the evolution of $g_t$.  At time
$t$, the relevant information is $g_t(\gamma(t))$ which we
write as $e^{2iU_t}$.  To compare radial $SLE_\kappa$ to chordal
or two-sided radial $SLE_\kappa$ with target boundary point
$w = e^{2i\theta}$, we also need to keep track of $g_t(w)
$ which we write as $ e^{2i\theta_t}$.

Radial $SLE_\kappa$ is obtained by solving the Loewner equation
with $a=2/\kappa$ and $U_t = - B_t$ a standard Brownian motion.  If
$X_t = \theta_t - U_t$, then $X_t$ satisfies
\[          dX_t = \beta \, \cot X_t \, dt + dB_t, \]
with $\beta = a$.
Much of the study of $SLE_\kappa$ in the radial parametrization
can be done by considering the SDE above.  In fact, the three
versions of $SLE_\kappa$ can be obtained by choosing different
$\beta$.  In Section
\ref{radbessec} we discuss the properties of this SDE
that we will need.  
We use the Girsanov theorem to estimate the Radon-Nikodym derivative
of the measures on paths for different values of $\beta$.

Section \ref{detersec} gives estimates for conformal maps that
will be needed.  The first two subsections discuss crosscuts
and the argument of a point.  If $D$ is a simply connected
subdomain of $\Disk$ containing the origin, then the intersection
of $D$ with the circle $\p \Disk_k$ can contain many components.
We discuss such crosscuts in Section \ref{crosscut} and state
a simple topological fact, Lemma \ref{jul19.lemma1},
 that is used in the proofs of
\eqref{apr4.100} and \eqref{apr4.1}.

 A classical conformally invariant measure
of distance between boundary arcs is extremal distance or
extremal length.  We will only need to consider distance between
arcs in a conformal rectangle for which it is useful
to estimate   
  harmonic measure, that is,   hitting
probabilities for Brownian motion. 
  We discuss the general strategy
for proving such estimates in Section \ref{extremesec}. 
The following subsections give specific estimates that will
be needed for radial and  two-sided radial.
The results in this section do not depend much at all on
the Loewner equation --- one fact that is used is that we are
stopping a curve at the first time it reaches  $\p \Disk_n$
for some $n$.
   The Beurling estimate (see \cite[Section 3.8]{LBook})
is the major tool for getting uniform estimates.

The main results of this paper can be found in Section \ref{slesec}.
The first three subsections define the three types of $SLE_\kappa$,
radial, chordal, two-sided radial, in terms of radial.  (To be more
precise, it defines these processes up to the time the path separates
the origin from the boundary point $w$).  Section \ref{lemmasec}
contains the hardest new result in this paper.  It is an analogue for 
to radial  $SLE_\kappa$ of a known estimate
for chordal $SLE_\kappa$ on the probability of hitting a set near
the boundary.  This is the main technical estimate for Theorem \ref{radmain}.
A different estimate is  proved in Section \ref{tworadsec} 
for two-sided radial.   The final section finishes
the proof Theorem \ref{maincorollary} by using a known technique
to show exponential rates of convergence. 

%
I would like to thank Dapeng Zhan for bringing up the fact that
this result is not in the literature and Joan Lind and Steffen Rohde
for useful
 conversations.

\subsection{Notation}

We let
\[ \Disk = \{|z| < 1\}, \;\;\;\;
        \Disk_n = e^{-n} \, \Disk = 
\{|z| < e^{-n}\}. \]
If $\gamma$ is a curve, then
\[   \rho_{n} = \inf\{t: \gamma(t) \in \p \Disk_n \}. \]
If $\gamma$ is random, then $\F_t$ denotes the $\sigma$-algebra
generated by $\{\gamma(s): s \leq t \}$ and $\G_n =
\F_{\rho_n}$ is the $\sigma$-algebra generated by
$\{\gamma(t): t \leq \rho_n\}. $ Let $D_t$ be the connected
component of $\Disk \setminus \gamma(0,t]$ containing
the origin and
\[             H_n = D_{\rho_n}. \]

If $D$ is a domain, $z \in D$, $V \subset \p D$, we let
$h_D(z,V)$ denote the harmonic measure starting at $z$, that
is, the probability that a Brownian motion starting
at $z$ exits $D$ at $V$.

When discussing $SLE_\kappa$ we will fix $\kappa$ and assume
that $0 < \kappa < 8$.  We let
\[   a = \frac 2 \kappa , \;\;\;\;\;  \alpha = \frac 8 \kappa -1
         = 4a -1 > 0 . \]

\section{Preliminaries}

\subsection{Radial Loewner equation}  \label{radsec}

Here we review the radial Loewner differential equation;
see \cite{LBook} for more details.
The radial Loewner equation describes the evolution 
of a curve from $1$ to $0$ in the unit disk $\Disk$.
Let $a > 0$,
and let $U_t:[0,\infty) \rightarrow \R$ be a continuous
function with $U_0=0$.  Let $g_t$ be the solution to the
initial value problem
\begin{equation}  \label{radloew}
   \p_t g_t(z) = 2a \, {g_t(z)}\, \frac{ e^{2iU_t}+ g_t(z)}
 { e^{2iU_t}- g_t(z)} 
 , 
\;\;\;\;  g_0(z) = z . 
\end{equation}
For each $z \in \overline \Disk \setminus \{1\}$, the solution
of this equation exists up to a time $T_z \in (0,\infty]$.
Note that $T_0 = \infty$
and $g_t(0) = 0$ for all $t$.  For each $t \geq 0$,  $D_t$,
as defined above,  equals  $\{z \in \Disk: T_z > t\}$, and $g_t$ is
the unique conformal transformation of  $D_t$
 onto $\Disk$
with $g_t(0) = 0, g_t'(0) > 0$. By differentiating
\eqref{radloew} with respect to $z$, we see that $\p_tg_t'(0) =
2a g_t'(0)$ which implies that $g_t'(0) = e^{2at}$.

If we define   $h_t(z)$  to be the continuous function of $t$ such
that
\[             g_t(e^{2iz}) = \exp\left\{2i h_t(z) \right\}, \;\;\;\; h_0(z) = z , \]
then the Loewner equation becomes
\begin{equation}  \label{cotloew}
    \p_t h_t(z) = a \, \cot(h_t(z) - U_t)  , \;\;\;\;
         h_0(z) = z . 
\end{equation}
We will consider this primarily for real $z = x \in (0,\pi)$. 
Note that if $x \in (0,\pi)$ and $D_t$ agrees with $\Disk$
in a neighborhood of  $e^{2ix}$, then
\begin{equation}  \label{jul13.1}
  |g_t'(e^{2ix})| = h'_t(x) . 
\end{equation}

The radial equation can also be used to study curves whose
``target'' point is a boundary point $w = e^{2i\theta_0},
0 < \theta_0 < \pi$. If we let $\theta_t = h_t(\theta_0)$,
 then 
  \eqref{cotloew} becomes
\[          \p_t \theta_t = a \, \cot(\theta_t - U_t) , \]
which is valid for $t < T_w$.  Using \eqref{jul13.1}, we get
\[   |g_t'(w)| = h_t'(\theta_0) = 
            \exp \left\{-a\int_0^t \frac{ds}{\sin^2(\theta_s -
             U_s)} \right\}.\]

\labove \textsf
{\Heuristic  \begin{small}  
 The radial Loewner equation as in \cite{Schramm}
or \cite{LBook} is usually written with $a=1/2$.  Also, the
$2U_t$ in the exponent in \eqref{radloew} is usually written
as $U_t$.  We choose to write $2U_t$ so that the equation
\eqref{cotloew} is simpler, and because 
$\theta_t - U_t$ corresponds to an
 angle when we map the  disk to the upper half
plane, see Section \ref{argsec}.
 \end{small}}
\lbelow

We say that $g_t$ is generated by   $\gamma$
if $\gamma:[0,\infty) \rightarrow \overline \Disk$ is
a curve such that for each $t$, $D_t$ is the connected
component of $\Disk \setminus \gamma(0,t]$ containing
the origin.  Not every continuous $U_t$ yields conformal
maps $g_t$ generated by a curve, but with probability
one $SLE_\kappa$  is generated by a curve (see \cite{RS}
for a proof for $\kappa \neq 8$ which is all that we need
in this paper).  For
ease, we will restrict our discussion to $g_t$ that are
generated by curves.

\begin{definition} $\;$

\begin{itemize}
\item  A curve arising from the Loewner equation
will be called a {\em Loewner curve}.  Two
such curves are equivalent if one is obtained from the other
by increasing reparametrization.
\item A Loewner curve has the $a$-radial parametrization
if  
$g_t'(0) = e^{2at}$.
\end{itemize}
\end{definition}

Recall that $\rho_n = \inf\{t: |\gamma(t)| = e^{-n}\}$.  A simple
conseqence of the Koebe $1/4$-theorem is the existence of
$c < \infty$ such that for all $n$
\begin{equation}  \label{newer}
   \rho_{n+1} \leq \rho_n + c . 
   \end{equation}

\subsection{Radial Bessel equation}   \label{radbessec}

Analysis of radial $SLE$ leads to studying
 a simple one-dimensional SDE \eqref{radbessel}
 that
we call the {\em radial Bessel equation}.  This
equation can be obtained using the Girsanov
theorem  by ``weighting'' or ``tilting'' a standard Brownian motion
as we now describe.
Suppose $X_t$ is a standard one-dimensional
Brownian motion defined on a probability
space $(\Omega,\Prob)$  with $0 < X_0
 < \pi$ and let $\tau = \inf\{t:  \sin X_t=
 0 \}.$  Roughly speaking, the radial Bessel equation
with parameter $\beta$ (up to time $\tau$)
is obtained by weighting the
Brownian motion locally by $(\sin X_t)^\beta$.  
 Since $(\sin X_t)^\beta$
  is not a local martingale, we need to compensate it by
a $C^1$ (in time) process $e^{\Phi_t}$ such that
$e^{-\Phi_t} \, (\sin X_t)^\beta$ is a local martingale.
The appropriate compensator is found easily using
It\^o's formula; indeed,   
\[   M_t =  M_{t,\beta} = (\sin X_t)^\beta \, e^{\beta^2t/2}
  \, \exp\left\{\frac{(1-\beta)\beta}2\int_0^t \frac{ds}
  {\sin^2 X_s} \right\}, \;\;\;\; 0 \leq t < \tau , \]
is a local martingale satisfying
\begin{equation}  \label{jan30.1}
         dM_{t} = \beta \, M_{t} \, \cot X_t \, dX_t. 
\end{equation}
In particular, for every $\epsilon > 0$ and
$t_0 < \infty$, there exists $C =   C(\beta,\epsilon,t_0) < \infty$
such that
 if $\tau_\epsilon =
 \inf\{t: \sin X_t \leq \epsilon \}$, then
\[    C^{-1}  \leq M_t \leq C
, \;\;\;\; 0 \leq t \leq 
  t_0 \wedge \tau_\epsilon.\]

Let $\Prob_\beta$
 denote the probability measure
on paths $X_t, 0 \leq t < \tau$ such that
for each $\epsilon > 0, t_0 < \infty$,
the measure $\Prob_\beta$ on paths $X_t,
0 \leq t \leq t_0 \wedge \tau_\epsilon$
is given by
\[         d\Prob_\beta = \frac{M_{t_0 \wedge \tau_\epsilon
   }}{M_0} \, d\Prob.\]
The Girsanov theorem states that
\[    B_t = B_{t,\beta} := X_t - \beta\int_0^t \cot X_s \, ds,
\;\;\;\; 0 \leq t < \tau \]
is a standard Brownian motion with respect to the measure
$\Prob_{\beta}$.  In other words,  
\[
   dX_t = \beta \, \cot X_t \,dt + dB_{t},
\;\;\;\;  0 \leq t < \tau .
\]
We call this the radial Bessel equation (with parameter
$\beta$).
 By comparison with the usual
Bessel equation, we can see that
\[    \Prob_{\beta}\{\tau = \infty \} = 1 \mbox{ if and only
if } \beta \geq \frac 12.\]

\labove \textsf
{\Heuristic  \begin{small}  
Although the measure $\Prob_\beta$
is mutually absolutely continuous with respect to $\Prob$ when one
restricts to curves $X_t, 0 \leq t \leq \tau_\epsilon \wedge t_0$,
it is possible that the measure $\Prob_\beta$
on curves $X_t, 0 \leq t <\tau \wedge t_0$ has a singular part
with respect to $\Prob$.
 \end{small}}
\lbelow

\subsubsection{An estimate}

Here we establish an estimate \eqref{may5.1}
for the radial Bessel equation
which we will use in the proof of continuity of two-sided
radial $SLE$.
Suppose that $X_t$ satisfies
\begin{equation}  \label{radbessel}
dX_t = \beta \, \cot X_t \,dt + dB_{t}, \;\;\; 0 \leq t < \tau , 
\end{equation}
where $\beta \in \R$, $B_t$ is a standard Brownian motion,
 and $\tau = \inf\{t:  \sin X_t = 0\}. $
Let
\[  F(x) = F_\beta(x) = \int_x^{\pi/2} (\sin t)^{-2\beta} \, dt ,
\;\;\;\; 0 < x < \pi,  \]
which satisfies
\begin{equation}  \label{may11.1}
  F''(x) + 2\beta \, (\cot x) \, F'(x) = 0.
\end{equation}
 
%

\begin{lemma} For every $\beta > 1/2$, there exists $c_\beta
< \infty$ such that if $0 < \epsilon < x  \leq \pi/2$, 
$X_t$ satisfies \eqref{radbessel} with $X_0 = x$, and 
\[           \tau_\epsilon
 = \inf\{t\geq 0: X_t = \epsilon \mbox{ or }
  \pi/2 \}, \]
then
\[ \Prob\{X_{\tau_\epsilon} = \epsilon\}
         \leq c_\beta \, (\epsilon/x)^{2\beta-1}. \] 
\end{lemma}

\begin{proof}  It\^o's formula and \eqref{may11.1}
show that  $F(X_{t \wedge \tau_\epsilon})$ is a
bounded
martingale, and hence  the optional sampling theorem implies that
\[   F(x) =  \Prob\{X_{\tau_\epsilon} = \epsilon\} \,
    F(\epsilon) + \Prob\{X_{\tau_\epsilon} = \pi/2\} \,
    F(\pi/2) = 
 \Prob\{X_{\tau_\epsilon} = \epsilon\} \,
    F(\epsilon) .\]
Therefore,
\[    \Prob\{X_{\tau_\epsilon}  = \epsilon\} = 
  \frac{F(x)}{F(\epsilon)}. \]
If $\beta > 1/2$,  then
\[    F(\epsilon)  \sim \frac{1}{2\beta - 1} \, \epsilon^{1-2\beta},
\;\;\;\;  \epsilon \rightarrow 0+,\]
from which the lemma follows.
\end{proof}

\begin{lemma}  \label{cinco.lemma}
 For every $\beta > 1/2, t_0 < \infty$,  there exists $c=
 c_{\beta,t_0}
< \infty$ such that if 
$X_t$ satisfies \eqref{radbessel} with $X_0 \in (0,\pi)$, then
\begin{equation} \label{may5.1}
 \Prob\left \{\min_{0 \leq t \leq t_0}
\sin X_t \leq \epsilon \sin X_0 \right\} \leq c \, \epsilon^{2\beta - 1}. 
\end{equation}
\end{lemma}

\begin{proof}  We allow constants to depend on $\beta,t_0$.  Let $r = \sin
X_0$.
  It suffices to prove the result when $r \leq 1/2$.  Let
$\sigma = \inf\{t: \sin X_t = 1 \mbox{ or } \epsilon  r\}$ and
let $\rho = \inf\{t > \sigma: \sin X_t = r\}. $  Using the
previous lemma we see that 
\[    \Prob\{\sin X_\sigma = \epsilon r\} \leq c \,
   \epsilon^{2 \beta - 1} . \]
Since $r \leq 1/2$ and there is positive probability that the process
started at $\pi/2$ stays in $[\pi/4, 3\pi/4]$ up to time $t_0$,
we can see 
  that 
\[           \Prob\{\rho > t_0 \mid  \sin X_\sigma = 1\} \geq c_1,
\]
Hence, if $q$ denotes the probability on the left-hand side  of \eqref{may5.1}, 
we get
\[   q \leq c \,  \epsilon^{2 \beta - 1} + \left(1 -
  {c_1} \right) \, q. \]
\end{proof}

\subsection{Deterministic lemmas}  \label{detersec}

 \subsubsection{Crosscuts in $\p \Disk_k$}$\;$ \label{crosscut}

\begin{definition}  A {\em crosscut} of a domain
$D$ is the image of a simple curve $\eta:(0,1) \rightarrow
D$ with $\eta(0+),\eta(1-) \in \p D$.  
\end{definition}

   Recall that  $H_n$ is the connected component of $\Disk \setminus
\gamma(0,\rho_n]$ containing the origin. Let
\[   \p_n^0 = \p H_n \setminus \gamma[0,\rho_n], \]
which is either empty or is an open subarc of $\p \Disk$.

 For each $0 < k < n$,
let $V_{n,k}$ denote the connected component of $H_n \cap \Disk_k$
that contains the origin, and let $\p_{n,k}=\p V_{n,k}
 \cap H_n$.     The connected components of
$\p_{n,k}$ comprise a collection $\arcs_{n,k} $
of open subarcs of $\p \Disk_k$.  Each arc
$l \in \arcs_{n,k} $ is a crosscut
of $H_n$ such that  $H_n \setminus l$ has two connected components.
Let $V_{n,k,l}$ denote the component of   $H_n \setminus l$ 
that does not contain
the origin; note that these components are disjoint
for distinct $l \in  \arcs_{n,k} $.
If $   \p_n^0 \neq \emptyset$, there is a unique arc 
 $l^*= l^*_{n,k} \in
\arcs_{n,k}$ such that $\p_n^0 \subset   \p V_{n,k,l^*}$.

\labove \textsf
{\Heuristic  \begin{small}  
Note that each $l \in \arcs_{n,k}$ is a connected component
of $\p \Disk _k \cap H_n$; however, there may be components
of $\p \Disk_k \cap H_n$ that are not  in
$\arcs_{n,k}$. In particular,
it is possible that $  V_{n,k,l} \cap \Disk_k \neq \emptyset.$
The arc $l^*$ is the unique arc in $\arcs_{n,k}$ such that
each path from $0$ to $\p_n^0$ in $H_n$ must pass through
$l^*$.  One can construct examples where there are other 
components $l$ of
$\p \Disk_k \cap H_n$ with the property that every
  path from $0$ to $\p_n^0$ in $H_n$ must pass through
$l$.  However, these components are not in $\arcs_{n,k}$.
 \end{small}}
\lbelow

If $k < n$ and $\gamma[\rho_n,\infty) \cap \p \Disk_k \neq \emptyset$,
then the first visit to $\p \Disk_k$ after  time $\rho_n$ must be to the
closure of the one of the crosscuts in ${\mathcal A}_{n,k}$. 
In this paper we will estimate the probability of hitting a given
crosscut.  Since there can be many crosscuts, it is not immediate how
to use this estimate to bound the probability of hitting any crosscut.
This is the technical issue that prevents us from extending \eqref{jul20.1}
to all $\kappa < 8$. 
The next lemma, however, shows that if the curve has not returned
to $\p \Disk_j$ after time $\rho_k$, then there is only one crosscut
in ${\mathcal A}_{n,k}$ from which one can access $\p \Disk_j$.

\begin{lemma} \label{jul19.lemma1}
Suppose $j < k < n$ and $\gamma$ is a Loewner
curve in $\Disk$ starting at $1$ with $\rho_n < \infty$, 
$H_n \not\subset \Disk_j$, and $
         \gamma[\rho_k,\rho_n] \subset \Disk_j .$
Then
there exists a unique crosscut $l \in {\mathcal A}_{n,k}$ such that
if $\eta:[0,1)\rightarrow H_n \cap \Disk_j$ is a simple curve
with $\eta(0) = 0,\eta(1-) \in \p \Disk_j$ and
\[           s_0 = \inf\{s: \eta(s) \in \p \Disk_{k}\}, \]
then $\eta(s_0) \in l$.  If $\p^0_n \neq \emptyset$, then  $l = l^*_{n,k}$.
\end{lemma}

\begin{proof}   Call $l \in \arcs_{n,k}$ {\em good} if there exists
a curve $\eta$ as above with $\eta(s_0)  \in l$.
Since $H_n \not\subset \Disk_j$, there exists at least one good
$l$. Also, if $-1 \in \p H_n$,
then $l^*_{n,k}$ is good.  Hence, we only need to
show there is at most one good $l \in \arcs_{n,k}$.
Suppose $\eta^1,\eta^2$ are two such curves with
times $s_0^1,s_0^2$ and let $z_j = \eta(s_0^j)$.
  We need to show that $z_1$ and $z_2$
are in the same crosscut  in ${\mathcal A}_{n,k}$. If
$z_1 = z_2 $ this is trivial, so assume
$z_1 \neq z_2$.  Let $l^1,l^2$ denote the two subarcs of
$\p \Disk_k$ obtained by removing $z_1,z_2$ (these are not
crosscuts in ${\mathcal A}_{n,k}$).   Let $l^1$ denote the
arc that contains $\gamma(\rho_k)$ and let $U$ denote the
connected component of $(H_k \cap \Disk_j)
\setminus (\eta^1 \cup \eta^2)$ that contains $\gamma(\rho_k)$.
Our assumptions imply that $\gamma[\rho_k,\rho_n] \subset
U$.  In particular, $l^2 \cap \gamma[\rho_k,\rho_n]
 = \emptyset$.  Therefore $l^2,z_1,z_2$ lie in the same
component of $H_n$ and hence in the same crosscut of ${\mathcal
A}_{n,k}$.
\end{proof}

\subsubsection{Argument}  \label{argsec}

\begin{definition}
  If $\gamma$ is a Loewner  curve in $\Disk$ starting
at $1$, $w \in \p \Disk \setminus \{1\}$, and $t < T_w$, then
\[      S_t = S_{t,0,w} =
 \sin \arg F_t(0) , \]
where $F_t: D_t \rightarrow \Half$ is a conformal transformation
with $F_t(\gamma(t)) = 0, F_t(w) =  \infty$.  
\end{definition}

 If $z \in \Half$, let $h_+(z)
= h_\Half(z,(0,\infty))$ denote the
probability that a Brownian motion starting at $z$ leaves $\Half$
at $(0,\infty)$ and let $h_-(z) = 1 - h_+(z)$ be the probability
of leaving at $(-\infty,0)$.  Using the explicit form of the Poisson
kernel in $\Half$, one can see that $h_-(z) = \arg(z)/\pi$. Using
this, we can see that
\[   S_0 = \sin \theta_0 \]
and 
\begin{equation} \label{may6.1}
      \sin \arg(z) \asymp \min\left\{h_+(z), h_-(z) \right\},
\end{equation}
where $\asymp$ means each side is bounded by an absolute constant
times the other side.

If $t < T_w$, we can write $\p D_t = \{\gamma(t)\} \cup \{w\}
 \cup \p_t^+ \cup \p_t^-$ where $\p_t^+$  ($\p_t^-$)  
is the part of $\p D_t$
that is sent to the positive (resp., negative) real axis by $F_t$.
Using conformal invariance and \eqref{may6.1}, we see that
\begin{equation}  \label{may6.1.new}
  S_t \asymp \min \left \{h_{D_t}(0, \p_t^+),
   h_{D_t}(0,\p_t^-)\right\}. 
\end{equation}

\subsubsection{Extremal length}  \label{extremesec}

The proofs of our deterministic lemmas will use
estimates of extremal length.  These can be
obtained by considering appropriate estimates for
Brownian motion which are contained in the next
lemma. 
  Let $\rect_L$
denote the open region bounded by a rectangle,
\[  \rect_L = \{x+iy \in \C: 0 < x < L, 0 < y < \pi \}. \]
We write $\p \rect_L = \p_0 \cup \p_l \cup
 \p^+_L \cup \p^-_L$ where  
\[   \p_0 = [0,i\pi],\;\;\;\;
  \p_L = [L, L+i\pi],\;\;\;\; \p^+_L = (i\pi,L + i \pi),
\;\;\;\; \p^-_L = (0,L). \]
If $D$ is a simply connected domain and $A_1,A_2$ are 
disjoint arcs on $\p D$, then the {\em $\pi$-extremal distance}
($\pi$ times the usual extremal distance or length) is the
unique $L$ such that there is a conformal transformation of
$D$ onto $\rect_L$ mapping $A_1,A_2$ onto $\p_0$ and
$\p_L$, respectively.   Estimates for the Poisson kernel
in $\rect_L$ are standard, see, for example, \cite[Sections 2.3 and
5.2]{LBook}.  The next two lemmas which we state without
proof give the estimates that we need. 

\begin{lemma}   \label{may4.lemma1}
There exist $0 < c_1 < c_2 < \infty$
such that the following holds.  Suppose $L \geq 2$,
and $V$ is the closed disk of radius $1/4$ about
$1 + (\pi/2)i$.
\begin{itemize}
\item If $z \in V$,
\begin{equation}  \label{jan25.0}
  c_1 \leq h_{\rect_L}(z,\p_0),
   h_{\rect_L}(z,\p^+_L ), 
    h_{\rect_L}(z,\p^-_L ) \leq c_2, 
\end{equation}
\item  If $z \in V$ and $A \subset \p_L$, then
\begin{equation}  \label{jan25.1}
   h_{\rect_L}(z, A) \leq c_2 \, e^{-L}\, |A|,
\end{equation}
where $|\cdot|$ denotes length.
\item If $B_t$ is a standard Brownian motion,
$\tau_L = \inf\{t: B_t \not \in \rect_L\},$
 $\sigma = \inf\{t: \Re(B_t) = 1\}$, then if
$0 < x < 1/2$ and $0 < y < \pi$,
\begin{equation}  \label{jan25.3}
   \Prob^{x+iy}\{B_\sigma \in V \mid \sigma< \tau_L\} \geq c_1. 
\end{equation}
\end{itemize}
\end{lemma}

\begin{lemma} \label{may18.newlemma}
 For every $\delta > 0$, there exists $c > 0$
such that if $L \geq \delta$ and $z \in \rect_L$ with $\Re(z)
\leq \delta/2$, then
\[     h_{\rect_L}(z, \p_L) \leq
     c \, e^{-L} \, \min\left\{ h_{\rect_L}(z, \p_L^+),
     h_{\rect_L}(z, \p_L^-)\right\}. \]
\end{lemma}

We explain the basic idea on how we will
 use these estimates.  Suppose
$D$ is a domain and $l$ is a crosscut of $D$ that divides
$D$ into two components $D_1,D_2$. Suppose $D_2$
is simply connected and    $A$ is a closed subarc
of $\p D_2 $ with $\p D_2 \cap \overline l = \emptyset.$ Let
$\p^+,\p^-$ denote the  connected
components of $\p D_2 \setminus\{\overline
l,A\}$.  We consider $A,\p^+,\p^-$ as arcs of $\p D$ in
the sense of prime ends.
Let $F: D_2 \rightarrow \rect_L$ be a conformal transformation
sending $l$ to $\p_0$ and $A$ to $\p_L$ and suppose
that $L \geq 2$. Let $l_1 =
 F^{-1}(1 + i(0,\pi)).$  Let $\tau = \inf\{t:B_t
\not \in D\}, \sigma = \inf\{t: B_t \in l_1\}$.  Then
if $z \in D_1$ and $A_1 \subset A$, 
\[           h_{D}(z, \p^+),
    h_{D}(z, \p^-)\geq c \,
          \Prob^z\{\sigma < \tau\}. \]
\[   h_{D}(z, A_1)\leq c \, 
\Prob^z\{\sigma < \tau\} 
 \,  e^{-L}\, |F(A_1)|. \]
In particular, there exists $c< \infty$ such that
for $z \in D_1, A_1 \subset A,$
\[     h_{D}(z, A_1)\leq c \, 
  e^{-L}\, |F(A_1)| \, \min\left\{ h_{D}(z, \p^+),
    h_{D}(z, \p^-)\right\}. \]

\subsubsection{Radial case}

We will need some lemmas that hold for all curves
$\gamma$ stopped at the first time they reach 
the sphere of a given radius or the first time
they reach a given vertical line.
 If $D$
is a domain and  $\eta:(0,1) \rightarrow D$ is
a crosscut, we
  write $\eta$ for the image $\eta(0,1)$ and
$\overline \eta = \eta[0,1]$.

\labove \textsf
{\Heuristic  \begin{small}  
  The next lemma is a lemma about Loewner curves, that is,
curves modulo reparametrization.  To make the statement nicer,
 we choose
a parametrization such that $\rho_{n+k} = 1$.  Although the
parametrization is not important, it is important that we
are stopping the curve at the first time it reaches $\p
\Disk_{n+k}$.
 \end{small}}
\lbelow


\begin{lemma}  \label{deterlemma}
There exists $c < \infty$ such that the following is
true.  Suppose $k > 0, n \geq 4$ and $\gamma:[0,1]
\rightarrow \Disk$ is a Loewner
curve with $\gamma(0) = 1;
|\gamma(1)| =e^{-n-k}$; and $e^{-n-k} < |\gamma(t)| 
< 1$ for $0 < t < 1$.  Let $D$ be the connected
component of  $ \Disk \setminus
\gamma(0,1]$ containing the origin, and let
\[ \eta = \{e^{-k + i\theta}: \theta_1 < \theta < \theta_2\}
\in \arcs_{n+k,k}
\]
 be a crosscut of $D$ contained in $\p \Disk_k$.

 Let $F: D \rightarrow \Disk$ be
the unique conformal transformation with $F(0) = 0,
F(\gamma(1)) = 1$.
Suppose that we write $\p \Disk$ as a disjoint union
\[     \p \Disk  = \{1\} \cup V_1 \cup V_2
  \cup V_3, \]
where $V_3$ is the closed interval of $\p \Disk$
 not containing
$1$
 whose endpoints are
the images under $F$ of $\eta(0+),
\eta(1-)$ and $V_1,V_2$
are connected, open intervals.  Then
\[    \diam[F(\eta)] \leq c \, e^{-n/2} \, (\theta_2-\theta_1) \,
        \min\left\{|V_1| , |V_2|\right\}, \]
where $|\cdot|$ denotes length.
\end{lemma}

\labove \textsf
{\Heuristic  \begin{small}  
 It is important for our purposes to show not only that
 $F(\eta)$ is small, but also that it is smaller than both
 $V_1$ and $V_2$.  When we apply the proposition, one
 of the intervals $V_1,V_2$ may be very small.
 \end{small}}
\lbelow

\begin{proof} 
Let $U$ denote the connected component of $D \setminus
\eta$ that contains the origin and note that $U$
is simply connected.
Let
\[ U^* = U \cap \left\{|z| > e^{-n-k} \right\}. \]
Since $\gamma(0,1) \subset \{|z| > e^{-n-k}\}$ and
$|\gamma(1)| = e^{-n-k}$, we can see that $U^*$ is
simply connected with  $\eta \cup \p
\Disk_{n+k} \subset \p U^*$. 
Let \[ g:  \rect_L   \longrightarrow
U^*\] be a conformal transformation mapping $\p_0$ onto
$\p \Disk_{n+k}$ and $\p_L
$ onto $\overline \eta$.  Such a transformation exists for
only one value of $L$, the $\pi$-extremal distance
  between
$\p \Disk_{n+k}$ and $\overline \eta$ in $U^*$.  Since
$\eta \cap \Disk_k = \emptyset$, and the complement of $U^*$
contains a curve connecting $\p \Disk_k$ and $\p \Disk_{n+k}$, 
  see that  $L \geq n/2 \geq 2$ (this can be done by comparison with
  an annulus, see. e.g., \cite[Example 3.72]{LBook}).    We write
\[   \p U^* = \p \Disk_{n+k} \cup \overline \eta \cup \p_-\cup 
\p_+ \]
where $\p_-$ ($\p_+$) is the image of $\p_L^-$
(resp., $\p_L^+$) under $g$.  Here we are considering
boundaries in terms of prime ends, e.g., if $\gamma$ is
simple then each point on
  $\gamma(0,1)$ corresponds to two points
in $\p D$.
 Note that $\{F(\p_-),F(\p_+)\}$ 
is $\{V_1, V_2\}$, so we can rewrite the conclusion of
the lemma  as
\begin{equation}  \label{apr6.5}
  h_U(0,\eta) \leq c \, e^{-n/2} \,(\theta_2 - \theta_1)
  \, \min\left\{ h_U(0,\p_-),  h_U(0,\p_+)
 \right\}. 
 \end{equation}

 Let $\ell = g(1 + i(0,\pi))$
which  separates $\p \Disk_{n+k}$
from $\eta$, and hence also separates the origin from $\eta$ in
$U$.  
 Let $B_t$ be a Brownian motion starting at the origin
 and
let 
\[ \sigma = \inf\{t: B_t \in \ell\}, \;\;\;\tau =
\inf\{t: B_t \not \in U\}. \]  
 
 Using conformal invariance and \eqref{jan25.0},
we can see that if $z \in
g(V)$, the probability that   a Brownian motion starting
at $z$ exits $U^*$ at $\p \Disk_{n+k}$ is at least $c_1$. However,
the Beurling estimate implies that this probability is bounded
above by $c \, [e^{-(n+k)} / |z|]^{1/2}. $  From
this we conclude that 
there exists $j$ such that $g(V)$ is contained in $\Disk_{n+k -j}$.
We claim that there exists
$c$ such that the probability that a Brownian motion
starting at $z \in D_{n+k-j}$ exits $U^*$ at $\eta$ is bounded 
above by $c
 e^{-n/2} \, (\theta_2- \theta_1)$.  Indeed, the Beurling estimate implies that
the probability to reach
$\p \Disk_{k+1}$ without leaving $U^*$ is 
$O(e^{-n/2})$, and using the Poisson kernel in the disk
we know that the probability
that a Brownian motion starting on $\p \Disk_{k+1}$ exits
$\Disk_k$ at $\eta$ is bounded above by $c\,(\theta_2 - \theta_1)$.  Using
\eqref{jan25.1}, we conclude that 
\[
\Prob\{B_\tau \in \eta \mid \sigma < \tau \} 
\leq c\, \max_{z \in g(V)} h_{U^*}(z,\eta)
  \leq c \, (\theta_t-\theta_1) \, e^{-n/2} . 
\]
Since the event $\{B_\tau \in \eta\}$ is contained in the event $\{\sigma < \tau\}$,
we see that
   \begin{equation}  \label{jan25.4}
        \Prob\{B_\tau \in \eta\} \leq c
                 \,  (\theta_t-\theta_1) \, e^{-n/2} \, \Prob\{\sigma < \tau\}. 
\end{equation}

%
%
Using \eqref{jan25.3} and conformal
invariance,  we can see that
\[ \Prob\{B_\sigma \in g(V) \mid \sigma < \tau\}
  \geq c_1, \]
and combining this with \eqref{jan25.0} we see that
\[   \min \left\{  \Prob\{B_\tau \in \p_- \mid \sigma < \tau\},
 \; \Prob\{B_\tau \in \p_+ \mid \sigma  < \tau\}\right\}
  \geq c_1^2. \]
  In particular, there exists $c$ such that
  \[    \min \left\{  \Prob\{B_\tau \in \p_-\} , \Prob\{B_\tau \in 
  \p_+\} \right\} \geq c \, \Prob\{\sigma < \tau\}. \]
Combining this with \eqref{jan25.4}, we get \eqref{apr6.5}.
\end{proof}

\subsubsection{An estimate for two-sided radial}   \label{moredetersec}

Recall that $\psi_{n,k}$ is the first time after $\rho_n$ that
the curve $\gamma$ intersects $l^*_{n,k}$, the crosscut defining
$V_{n,k}^*$. 
%
%
%

\begin{lemma}   \label{may18.lemma1}
 There exists $c < \infty$ such that if
$0 < k < n$ and $\psi = \psi_{n,k}^* < \infty$, then
\[    S_{D_\psi}(0) \leq  c\, e^{(k-n)/2}\, S_{H_n}(0) \]
\end{lemma}

\begin{proof}  
Let $\eta = l^*_{n,k}$  and let
 $U^*$ be as in the proof of Lemma
\ref{deterlemma}.  Since $\eta$ disconnects $-1$ from $0$,
we can see that when we write
\[  \p U^*  = \p \Disk_{n+k} \cup \overline \eta \cup
             \p_- \cup \p_+, \]
then $\p_- \subset \p_- U, \p_+ \subset \p_+U$ (or the other
way around).  We also have a universal lower bound on
$ h_{H_n \setminus \eta}
          (0,\eta)$.  Hence from
Lemma \ref{may18.newlemma} and \eqref{may6.1.new}
we see that
\[   h_{H_n \setminus \eta}
          (0,\eta) \leq c \, e^{(k-n)/2}\,  S_{H_n}(0).\]

There is a crosscut $l$ of $D_{\psi}$
that  is contained in $l^*$, has one of its endpoints
equal to $\gamma(\psi)$, and such that $0$ is disconnected
from $-1$ in $D_{\psi}$ by $l$. If $V$ denotes the
connected component of $D_{\psi_n} \setminus l$ containing
the origin, then  $\p V
\cap \p D_{\psi}$ (considered as prime ends) is contained
in either $\p_+ D_{\psi}$ or $\p_- D_\psi$.   Therefore,
\[ 
  S_{D_\psi}(0)  \leq  c \,  h_{D_{\psi} \setminus l}(0,
l)  
  \leq c \, 
  h_{H_n \setminus \eta}
          (0,\eta) \leq c \, e^{(k-n)/2}\,  S_{H_n}(0).
\]
\end{proof}

%
%

\labove \textsf
{\Heuristic  \begin{small}  
  This proof uses strongly the fact that $l^*$ separates $-1$
  from $0$ in $H_n$.   The reader may wish to draw some
pictures to see that for other crosscuts $l \in \arcs_{n,k}$,  
$S_{D_\psi}(0)$  need not be small.
 \end{small}}
\lbelow

\section{Schramm-Loewner evolution ($SLE$)}  \label{slesec}

Suppose $D$ is a simply connected domain with two distinct
boundary points $w_1,w_2$ and one interior point $z$.  There
are three closely related versions of $SLE_\kappa$
in $D$: chordal
$SLE_\kappa$ from $w_1$ to $w_2$; radial from $w_1$ to $z$;
and two-sided radial from $w_1$ to $w_2$ going through $z$.
The last of these can be thought of as chordal $SLE_\kappa$
from $w_1$ to $w_2$ conditioned to go through $z$.  All
of these processes are conformally invariant and are  
defined only up to increasing reparametrizations.  Usually
chordal $SLE_\kappa$ is parametrized using a ``half-plane''
or ``chordal'' capacity with
respect to $w_2$ and radial and two-sided radial $SLE_\kappa$
are defined with a radial parametrization with respect to
$z$, but this is only a convenience.   If the same parametrization
is used for all three processes, then they
are   mutually absolutely continuous with each other if one
stops the process at a time before which that paths separate
$z$ and $w_2$ in the domain.

We now give precise definitions.  For ease we will choose
$D = \Disk$, $z=0$, $w_1 = 1$ and $w_2 = w = e^{2i\theta_0}$
with $0 < \theta_0 < \pi$.  We will use a radial
parametrization.  We first define radial $SLE_\kappa$ (for which
the point $w$ plays no role in the definition) and then
define chordal $SLE_\kappa$ (for which the point $0$ is irrelevant
when one considers processes up to reparametrization but is
important here since our parametrization depends on this point)
and two-sided $SLE_\kappa$ in terms of radial.  The definition
using the Girsanov transformation is really just one example
of a general process of producing ``$SLE(\kappa,\rho)$
processes''.

  Let $h_t(x)$ be the solution of \eqref{cotloew} with
$h_0(x) = \theta_0$ and let
\begin{equation}  \label{feb4.1}
      X_t = h_t(w) - U_t , \;\;\;\; S_t = \sin X_t.
\end{equation}
Note that $S_t$ is the same as defined in
Section \ref{argsec} and
\begin{equation}  \label{jan30.0}
 h_t'(w) =  \exp\left\{-a\int_0^t \frac{ds}{S_s^2} \right\}. 
\end{equation}
Let
\[  \tau_\epsilon = \tau_\epsilon(w) =
 \inf\{t \geq 0: S_t = \epsilon\} , \]
\[        \tau = \tau_0(w) = \inf\{t\geq 0: S_t = 0 \}
 = \inf\left\{t: \dist(w,\Disk \setminus D_t) = 0
 \right\}. \]
 
\subsection{Radial $SLE_\kappa$}

If $\kappa >0$, then {\em radial $SLE_\kappa$ (parametrized
so that $g_t'(0) = e^{4t/\kappa}$)} is the solution of the
Loewner equation \eqref{radloew} or \eqref{cotloew}
with $a = 2/\kappa$ and
 $U_t = -B_t$ where $B_t$ is a
standard Brownian motion.  This definition does not reference
the point $w$.  However, if we define $X_t$ by
\eqref{feb4.1}, we have  
\[            dX_t = a \, \cot X_t \, dt + d B_t.\]

  Suppose that $(\Omega,\F,\Prob_0)$
is a probability space under which $X_t$ is a Brownian
motion.  Then, see Section \ref{radbessec},
 for each $\beta \in \R$ there is
a probability $\Prob_\beta$  such that
\[       B_{t,\beta} = X_t - \beta \int_0^t \cot X_s \, ds
,
\;\;\;\; 0 \leq s < \tau , \]
is a standard Brownian motion. In other words,
\[      dX_t = \beta \, \cot X_t \, dt + dB_{t,\beta}. \]
 In particular, $B_t =
B_{t,a}$.  
We call this radial $SLE_\kappa$ {\em weighted locally by
$S_t^{\beta - a}$}, where $S_t = \sin X_t$.
Radial $SLE_\kappa$ is obtained by choosing $\beta = a$. 
Using \eqref{jan30.0}
we can write the local martingale in \eqref{jan30.1}
as
\[  M_{t,\beta} = S_t^{\beta} \, e^{t\beta^2/2}
  \, h_t'(w)^{\frac{\beta(\beta -1)}{2a}}.\]

We summarize the discussion in Section \ref{radbessec}
as follows.  If $\sigma$ is a stopping time, let
$\F_\sigma$ denote the $\sigma$-algebra generated
by $\{X_{s \wedge \sigma}: 0 \leq s < \infty\}$.

\begin{lemma}  Suppose $\sigma$ is a  stopping
time with $\sigma \leq 
\tau_\epsilon$ for some $\epsilon > 0$.
Then
the measures
  $\Prob_\alpha$ and $\Prob_\beta$  are mutually
absolutely continuous
 on $(\Omega,\F_\sigma)$.
 More precisely, 
  if $t_0 < \infty$, there exists $c =
c(\epsilon, t_0,\alpha,\beta) < \infty$ such that if $\sigma \leq
\tau_\epsilon \wedge t_0$, 
\begin{equation}\label{feb4.3}
   \frac 1c \leq  \frac{d\Prob_\alpha}
    {d \Prob_\beta}  \leq c .
    \end{equation}
 \end{lemma}

 \labove \textsf
{\Heuristic  \begin{small}  
Clearly we can give more precise estimates for the Radon-Nikodym derivative,
but this is all we will need in this paper.
 \end{small}}
\lbelow

Different values of $\beta$ given different processes;
chordal and two-sided radial $SLE_\kappa$ correspond
to particular values.

 \subsection{ Chordal $SLE_\kappa$: $\beta = 1-2a$}
{\em Chordal
$SLE_\kappa$ (from $1$ to $w$ in $\Disk$ in the radial
parametrization stopped at time $T_w$)} is obtained from
radial $SLE_\kappa$ by weighting
locally by   $S_t^{1-3a}$.
In other words,
\begin{equation}  \label{chordSDE}
   dX_t =  (1-2a) \, \cot X_t \, dt + dB_{t,1-2a}, 
\end{equation}
where $B_{t,1-2a}$ is a Brownian motion with respect to
$\Prob_{1-2a}$.

This is not the usual way chordal $SLE_\kappa$ is defined so let
us relate this to the usual definition.  $SLE_\kappa$ from $0$
to $\infty$ in $\Half$ is defined by considering the Loewner
equation
\[             \p_t g_t(z) = \frac{a}{g_t(z) - U_t} ,\]
where $U_t = -B_t$ is a standard Brownian motion.  There
is a random curve $\gamma:[0,\infty) \rightarrow \overline
\Half$ such that the domain of $g_t$ is the unbounded component
of $\Half \setminus \gamma(0,t]$.  $SLE_\kappa$ connected
boundary points of other simply connected domains is defined
(modulo time change) by conformal transformation.  One can use
It\^o's formula to check that our definition agrees (up to
time change) with the usual definition.

If $D$ is a simply connected domain and $w_1,w_2$ are boundary
points at which $\p D$ is locally smooth, the chordal
$SLE_\kappa$
partition function is defined (up to an unimportant multiplicative
constant) by
\[         H_\Half(x_1,x_2) = |x_2-x_1|^{-2b}, \]
and the scaling rule
\[       H_D(w_1,w_2) = |f'(w_1)|^b \, |f'(w_2)|^b
  \, H_{f(D)}(f(w_1),f(w_2)) , \]
where $b=(3a-1)/2$ is the boundary scaling exponent.  
To obtain $SLE_\kappa$ from $0$ to $x$ in $\Half$ one can take
$SLE_\kappa$ from $0$ to $\infty$ and then weight locally by the
value of the partition function between $g_t(x)$ and $U_t$, i.e.,
by $|g_t(x) - U_t|^{-2b}$.   A simple computation
shows that
\[   H_\Disk(e^{2i\theta_1},e^{2i\theta_2})
   =   |\sin (\theta_1 - \theta_2)|^{-2b}
 =|\sin (\theta_1 - \theta_2)|^{1-3a}.\]  
Hence we see that chordal $SLE_\kappa$ in $\Disk$ is obtained
from radial $SLE_\kappa$ by weighting locally by the chordal
partition function.

\subsection{Two-sided radial $SLE_\kappa$: $\beta = 2a$} \label{twosec}

If $\kappa < 8$,  {\em
Two-sided radial $SLE_\kappa$ (from $1$ to $w$ in $\Disk$ going
through $0$ stopped when it reaches $0$)} is obtained by
weighting chordal $SLE_\kappa$ locally by $(\sin X_t)^{(4a-1)}$.
Equivalently, we can think of this as weighting
radial $SLE_\kappa$ locally by $(\sin X_t)^a$.
It should be considered as chordal $SLE_\kappa$ from $1$
to $w$ conditioned to go through $0$.  

\labove \textsf
{\Heuristic  \begin{small}  
If $\kappa \geq 8$, $SLE_\kappa$ paths are plane-filling and hence
conditioning the path to go through a point is a trivial conditioning.
For this reason, the discussion of two-sided radial is restricted
to $\kappa < 8$. 
 \end{small}}
\lbelow

The definition comes from the Green's function for chordal
$SLE_\kappa$.  If $\gamma$ is a chordal $SLE_\kappa$ curve
from $0$ to $\infty$ and $z \in \Half$, let $R_t$ denote
the conformal radius of the unbounded component of $\Half \setminus
\gamma(0,t]$ with respect to $z$, and let
$R = \lim_{t \rightarrow \infty} R_t.$ The Green's function
$G(z) = G_\Half(z;0,\infty)$ can
be defined (up to multiplicative constant) by the relation
\[               \Prob\{R_t \leq \epsilon\}
    \sim c\,  \epsilon^{2-d} \, G(z), \;\;\;\;
  z \rightarrow \infty. \]
where $d = \max\{1 + \frac \kappa 8,2\}$ is the Hausdorff dimension
of the paths. Roughly speaking, the probability that a chordal
$SLE_\kappa$ in $\Half$ from $0$ to $\infty$ gets within distance
$\epsilon$ of $z$ looks like $c \, G(z) \, \epsilon^{2-d}$.
 For other simply connected domains, the Green's
function is obtained by conformal covariance
\[            G_D(z;w_1,w_2) =
    |f'(z)|^{2-d}
  \, G_{f(D)}(f(z);f(w_1),f(w_2)), \]
assuming smoothness at the boundary.  In particular, one can show that
(up to an unimportant multiplicative constant)
\[     G_\Disk(0;1,e^{i\theta}) = 
     (\sin \theta)^{4a-1}  , \;\;\;\; \kappa < 8.
 \]
 

 \section{Proofs of main results}

\subsection{Continuity of radial $SLE$}  \label{lemmasec}

The key step to proving continuity of radial $SLE_\kappa$
is an extension of  an estimate
for chordal $SLE_\kappa$ to radial $SLE_\kappa$. 
The next lemma gives the analogous estimate for chordal
$SLE_\kappa$; a proof can be found in \cite{AK}.
Recall that $\alpha = (8/\kappa) - 1$.

\begin{lemma}  For every $ 0 < \kappa < 8$, there
exists $c < \infty$ such that  if
 $\eta$ is a crosscut in $\Half$ and $\gamma$
is a chordal  $SLE_\kappa$ curve
 from $0$ to $\infty$ in $\Half$, then
\begin{equation} \label{chordest}
 \Prob\{\gamma(0,\infty) \cap \eta \neq \emptyset\} \leq 
   c \,\left[
\frac{\diam ( \eta)}{\dist(0,\eta)}\right]^{
\alpha}. 
\end{equation} 
\end{lemma}

We will prove the corresponding result
for radial $SLE_\kappa$.
We
start by establishing the estimate up to a fixed time
(this is the hardest estimate), and then extending
the result to infinite time.

\begin{lemma}  \label{groundhog.1}
 For every $t < \infty$, there exists $C_t <\infty$
such that the following holds. 
Suppose $\eta$ is a crosscut of $\Disk$
and   $\gamma$ is a radial $SLE_\kappa$ curve
from $1$ to $0$ in $\Disk$. Then
\[  \Prob\{\gamma(0,t] \cap \eta \neq \emptyset \}
  \leq C_t \, 
          \left[\frac{\diam (\eta)}{\dist(1,\eta)}\right]^{
\alpha}.\]
\end{lemma}

\begin{proof} 
Fix a positive integer $n$ sufficiently
large so
that $\gamma(0,t] \cap \Disk_n = \emptyset$. All
constants
in this proof may depend on $n$ (and hence on $t$).

Since $\dist(1,\eta) \leq 2$,  it suffices to prove the lemma 
for crosscuts satisfying
$\diam (\eta) < 1/100$ and $\dist(1,\eta) > 100 \, \diam (\eta)$.
Such crosscuts do not disconnect $1$ from $0$ in $\Disk$.

Let $V = V_\eta$ denote the connected component of $\Disk
\setminus \eta$ containing the origin, and let $F= F_\eta$
be a conformal transformation of $V$ onto $\Disk$
with $F(0) = 0$.
 We write
$\p V$ as a disjoint union:
\[  \p V = \{1\} \cup \eta[0,1] \cup \p_1 \cup \p_2 , \]
where $\p_1,\p_2$ are open connected subarcs
of $\p \Disk$.  Let
\[   L(\eta) =  \frac 1{2\pi} \, |F(\eta)| =
  h_{V_\eta}(0,\eta), \]
\[  
  L^*(\eta) =  \frac 1{2\pi}\, \min\left\{|F(\p_1)|,
 |F(\p_2)|\right\} = \min\left\{h_{V_\eta}(0,\p_1), h_{V_\eta}
  (0,\p_2) \right\},\]
where $|\cdot |$ denotes length.  Note that
\[    \diam(\eta) \asymp L(\eta)  \;\;\;\;\;
   \dist(1,\eta) \asymp L^*(\eta), \]
and hence we can write the conclusion of
the lemma as 
\begin{equation}  \label{jul10.5}
 \Prob\{\gamma(0,t] \cap \eta \neq \emptyset\} \leq 
   c \,\left[
\frac{L ( \eta)}{L^*(\eta)}\right]^{
\alpha},
\end{equation}
which is what we will prove.

Let $\gamma$ be a radial $SLE_\kappa$ curve.  If  
  $\gamma(0,t] \cap \eta = \emptyset$ and $\eta(0,1)
\subset D_t$,  let $V_t$ be the connected component
of $D_t \setminus \eta$ containing the origin
with corresponding maps $F_t$. We write
\[   \p V_t = \{\gamma(t)\} \cup \eta[0,1] \cup \p_{1,t}
  \cup \p_{2,t} \]
where the boundaries are considered in 
terms of prime ends.  
Let 
\[  L_t(\eta) =   \frac 1{2\pi}\, |F_t(\eta)|
  = h_{V_t}(0,\eta), \]
\[  L_t^*(\eta) =  \frac 1{2\pi}\,  \min\left\{|F_t(\p_{1,t})|,
 |F_t(\p_{2,t})|\right\} = \min\left\{h_{V_t}(0,\p_{1,t}),
 h_{V_t}
  (0,\p_{2,t}) \right\}. \] Note
that $L_t(\eta)$ decreases with $t$ but
$L_t^*(\eta)$ is not monotone in $t$.

As before, let $\rho = \rho_n$ be the first time $s$
that $|\gamma(s)| \leq e^{-n}$; our assumption on
$n$ implies that $\rho \geq t$.  Let $\sigma =
\sigma_n$ be the first time $s$ that 
  $\Re[\gamma(s)] \leq e^{-2n}$.   Our proof will include
a series of claims each of which will be proved after their
statement.

\begin{itemize}
\item {\bf Claim 1}. There exists 
 $u > 0$ (depending on $n$), such that
\begin{equation}  \label{feb6.1}
      \Prob\{\sigma \wedge \rho \geq t \}
  \geq u . 
\end{equation}
\end{itemize}

Deterministic estimates using the Loewner
equation
 show that  if $U_t$ stays sufficiently close to $0$, then
  $\rho < \sigma$. 
Therefore, since $\rho \geq  t$,
\[  \Prob\{  \sigma \wedge \rho > t\}
  \geq   \Prob\{ \rho < \sigma\} > 0. \]
 
 \medskip

\begin{itemize}
\item {\bf Claim 2}.
There exists $c < \infty$ such that 
\begin{equation}  \label{feb4.2}
   \Prob\{\gamma(0,\sigma] \cap \eta \neq \emptyset\}
   \leq 
 c \,\left[
\frac{\diam (\eta)}{\dist(0,\eta)}\right]^{
\alpha} 
\end{equation}
\end{itemize}

To show this we compare radial $SLE_\kappa$ from
$1$ to $0$ with chordal $SLE_\kappa$ from $1$ to $-1$. 
 Note by \eqref{newer} that 
   $\sigma$ is uniformly bounded. Straightforward
geometric  arguments show
that there exists $c $ (recall that constants
may depend on $n$)
such that
$c^{-1} \leq h_\sigma'(-1) \leq c$ and $\sin X_\sigma
\geq c^{-1}$.
 By \eqref{feb4.3} the  Radon-Nikodym
derivative of radial $SLE_\kappa$ with respect
to  chordal $SLE_\kappa$
 is uniformly bounded away from $0$ 
and $\infty$ and therefore
if $\tilde \gamma$ denotes a chordal $SLE_\kappa$ path from
$1$ to $-1$,
\[  \Prob\{\gamma(0,\sigma] \cap \eta = \emptyset \}
   \asymp \Prob\{\tilde \gamma(0,\sigma] \cap \eta = \emptyset\}. \]
Hence \eqref{feb4.2} follows from \eqref{chordest}.

\medskip 

\begin{itemize}
\item
{\bf Claim 3}.
There exists $\delta > 0$
such that if $L(\eta), L^*(\eta) \leq
 \delta$, then
on the event 
\begin{equation} \label{topcon1}
  \{\gamma(0,\sigma] \cap \eta = \emptyset\}, 
\end{equation}
we have 
\begin{equation}  \label{import}
  \frac{L_\sigma(\eta)}{L^*_\sigma(\eta)}
          \leq \frac{L(\eta)}{L^*(\eta) }.
\end{equation}
\end{itemize}

It suffices to consider $\eta$ with $L(\eta), L^*(\eta)
\leq 1/10$, and without
 loss of generality we assume that 
$\eta$ is ``above'' $1$ in the sense that its endpoints
are $e^{i\theta_1},e^{i\theta_2}$ with $0 < \theta_1 \leq
\theta_2 < 1/4$.  
Let $\gamma(\sigma) = e^{-2n} + i y , $ and let
$V_\sigma = V_{\eta,\sigma}$ be the connected component of $V
\setminus \gamma(0,\sigma]$ containing the origin.  Suppose
$\gamma(0,\sigma] \cap \eta = \emptyset$ and 
\begin{equation}  \label{topcon2}
\eta \subset V_\sigma, 
\end{equation}
(If $\eta \not\subset V_\sigma$, then $L_\sigma(\eta) = 0$.)
As before
we write
\[   \p V_\sigma = \{\gamma(\sigma)\} \cup \eta[0,1] \cup
\p_{1,\sigma} \cup \p_{2, \sigma} , \]
where we write $\p_{1,\sigma}$ for the component of the
boundary that includes $-1$.  Note that \eqref{topcon1}
and \eqref{topcon2} imply that  $\p_{1,\sigma}$
in fact contains $\{e^{i\theta}: \theta_2 < \theta <
3\pi/2\}$.
 Let $\ell$
denote the crosscut of $V_\sigma$ given by the vertical
line segment whose
lowest point is $\gamma(\sigma)$ and whose
highest point is on $\{e^{i\theta}: 0 < \theta < \pi/2\}$. 
 Note that $V_\sigma \setminus \ell$
has  two connected components,
one containing the origin and the other, which we denote
by $V^* = V^*_{\sigma,\eta}$, with $\eta \subset \p V^*$.
Let $\epsilon$ denote the length of $\ell$ and for the moment
assume that $\epsilon < 1/4$.  Topological considerations
using \eqref{topcon1} and \eqref{topcon2} imply that
all the points in $\p V^* \cap \{z \in \Disk:
e^{-2n} < \Re(z) < e^{-1} \}$ (considered as prime ends) are
in $\p_{2,\sigma}$. 

We consider another crosscut $\ell'$ defined as follows.
Let $x = e^{-2n} + \epsilon$.
Start at  $x + i \, \sqrt{1 - x^2}
  \in \p \Disk$ and take a vertical segment downward of length
$2\epsilon$ to the point $z' = x + i \, 
(\sqrt{1 - x^2}-2\epsilon).$  From $z'$ take a horizontal
segment to the left ending at $\{\Re(z) = e^{-2n}\}$.  This
curve, which is a 
concatentation of two line
 segments, must intersect the path $\gamma(0,\sigma]$ at
some point;  let $\ell'$  be the crosscut 
obtained by stopping this curve at the first
such intersection.
Let $V'$ be the connected component of $V^*\setminus \ell'$
that contains $\ell$ on its boundary.
The key observations are: $\dist(\ell,\ell') \geq c \epsilon,$
$\diam(\ell) = O(\epsilon), \diam (\ell')  = O(\epsilon)$ and
${\rm area}(V') = O(\epsilon^2)$.
 In particular
(see, e.g., \cite[Lemma 3.74]{LBook}) the $\pi$-extremal distance
between $\ell$ and $\ell'$ is bounded below by a positive
constant $c_1$ independent of $\epsilon$.

Let $B_t$ be a standard complex Brownian motion starting
at the origin and let
\[  \tau = \inf\{t: B_t \not\in V_\sigma\}, \;\;\;\;
   \xi = \inf\{t: B_t \in \ell'\}.\] 
Then using
Lemma \ref{may4.lemma1}
we see that
\begin{equation}  \label{may10.1}
     \min \left\{  \Prob\{B_\tau \in \p_{1,\sigma}\}
  , \Prob\{B_\tau \in \p_{2,\sigma}\}\right \} \geq c_2
\, \Prob\{\xi < \tau \}. 
\end{equation}
Also, we claim that
\[        \Prob\{B_\tau \in \eta \mid \xi < \tau\}
   \leq c \, \epsilon \, L(\eta) .\]
To justify this last estimate, note that $\dist(B_\xi,\p \Disk)
 = O(\epsilon)$.   It suffices to consider the
probability that a Brownian motion starting at $B_\xi$
hits $\eta$ before leaving $\Disk$.
 The gambler's ruin estimate implies
that the probability that a Brownian motion
starting at $B_\xi$
reaches $\{\Re(z) = 1/2\}$ before leaving $\Disk$
is $O(\epsilon)$.  Given that we reach  $\{\Re(z) = 1/2\}$,
the probability to hit $\eta$ before leaving $\Disk$ is
$O(L(\eta))$. 
Therefore,
\begin{equation}  \label{may10.2}
   \Prob\{B_\tau \in \eta\} \leq c \, 
  \epsilon \, L(\xi)\, \Prob\{\xi < \tau \}.
\end{equation}
By combining \eqref{may10.1} and \eqref{may10.2}, we see
  that we can choose $\epsilon_0 > 0$
such that for $\epsilon < \epsilon_0$,
\[    \Prob\{B_\tau \in \eta\} \leq
    L(\eta) \,   \min \left\{  \Prob\{B_\tau \in \p_{1,\sigma}\}
  , \Prob\{B_\tau \in \p_{2,\sigma}\} \right\} , \]
and hence
\begin{equation}  \label{jul10.1}
   L_\sigma(\eta) \leq L(\eta) \, L^*_\sigma(\eta)\;\;\;
  \mbox{ if } \;\;\; \epsilon \leq \epsilon_0.
\end{equation}

One we have fixed $\epsilon_0$, we note
there exists $c = c(\epsilon_0) > 0$ such that if $\epsilon
\geq \epsilon_0$,
\[ L_\sigma^*(\eta) =  \min \left\{  \Prob\{B_\tau \in \p_{1,\sigma}\}
  , \Prob\{B_\tau \in \p_{2,\sigma}\}\right \} \geq c_2
. \]   Indeed, to bound $\Prob\{B_\tau \in \p_{1,\sigma}\}$ from
below we consider Brownian paths starting at the origin that
leave $\p \Disk$ before reaching $\{z: \Re(z) \geq e^{-2n}\}$.
To bound $\Prob\{B_\tau \in \p_{2,\sigma}\}$  consider Brownian
paths in the disk that start at the origin, go through the crosscut
$l$ (defined using $\epsilon = \epsilon_0$), and then make a 
clockwise loop about $\gamma(\sigma)$ before leaving $\Disk$
and before reaching $\{\Re(z) \geq 1/2\}$.  Topological considerations
show that these paths exit $V_\sigma$ at $\p_{2,\sigma}$.
Combining this with \eqref{jul10.1} and the estimate
$L_\sigma(\eta) \leq L(\eta)$, we see that there
exists $c_1 >0 $ such that for all $\eta$,
\[   L_\sigma(\eta) \leq c_1^{-1}\, L(\eta) \, L^*_\sigma(\eta).\]
In particular, 
\[   \frac{L_\sigma(\eta)}{L^*_\sigma(\eta)}
   \leq  \frac{L(\eta)}{L^*(\eta)} \;\;\; \mbox{ if }
\;\;\; L^*(\eta) \leq c_1.\]\
From this
we conclude \eqref{import}. 
 
 \medskip

Fix $\delta$ such that \eqref{import} holds, and  let 
 $\phi(r)$ be the supremum of
\[         \Prob\{\gamma(0,t] \cap  \eta \neq
\emptyset \} \]
where the supremum is over all $\eta$ with 
\[ L(\eta)
  \leq r \, \min\{\delta,L^*(\eta)\}. \]

\medskip

\begin{itemize}
\item {\bf Claim 4}. 
If $r  < \delta$,
 then  $\phi(r)$ equals $\tilde \phi(r)$
which is defined to be  the
 supremum of
\[         \Prob\{\gamma(0,t] \cap  \eta \neq
\emptyset \} \]
where the supremum is over all $\eta$ with 
\[ L(\eta)  
 \leq \min\{\delta, r \, L^*(\eta)\}
 .\]
\end{itemize}

To see this, suppose $\eta$ is a curve with
$L(\eta) \leq r\delta,
 L^*(\eta) >\delta.$  Let $S$
be the first time $s$ such that $L_s^*(\eta) = \delta$.
Note that
$   S < \inf\{s: \gamma(s) \in \eta \}. $
Since $L_S(\eta) \leq L(\eta) \leq r\delta$, we see that
\[   \Prob\{\gamma(0,t] \cap \eta \neq \emptyset
\}
  \leq  \Prob\{\gamma(0,t] \cap \eta \neq \emptyset
 \mid S < \infty\} \leq \tilde \phi(r) . \]
This establishes the claim.

\medskip

To finish the proof of the lemma, suppose $r < \delta$.
Since $\tilde \phi(r) = \phi(r)$, we can find
a crosscut $\eta$  with $
 L(\eta) \leq r\, L^*(\eta)$ and $L^*(\eta) \leq \delta$ 
 such that
\[  \phi(r) = \Prob\{\gamma(0,t] \cap \eta
 \neq \emptyset\} .\]
(For notational ease we 
are assuming the supremum is obtained. 
We do not need to assume this, but could rather take
a sequence of crosscuts $\eta_j$ with
$ \Prob\{\gamma(0,t] \cap \eta_j
 \neq \emptyset \} \rightarrow \phi(r)$.)  
 Using Claim 3, we see that if
$\gamma(0,\sigma] \cap \eta = \emptyset$, then
\[   L_\sigma(\eta) \leq L(\eta) \leq r, \;\;\;\;
     L^*_\sigma(\eta) \leq L_\sigma(\eta). \]
Therefore,
\[   \Prob\{\gamma(0,t] \cap \eta \neq \emptyset \mid
           \gamma(0,\sigma \wedge \rho] \cap \eta= \emptyset \}
        \leq \phi(r) . \]
Hence, using \eqref{feb6.1},
\[  \phi(r) =  \Prob\{\gamma(0,t] \cap \eta  \neq \emptyset\}
\leq \Prob\{\gamma(0,\sigma \wedge \rho]
  \cap \eta  \neq \emptyset \} + (1-u) \, \phi(r), \]
which implies
\[  \phi(r) \leq \frac{1}{1-u} \, \Prob\{\gamma(0,\sigma \wedge \rho]
  \cap \eta  \neq  \emptyset \}.\]
 Combining this with \eqref{feb4.2}, we get
\begin{equation}  \label{feb6.2}
      \phi(r) 
    \leq \frac{c}{1-u} \, \left[\frac{L(\eta)}{L^*(\eta)}
   \right]^{\alpha} \leq c' \, r^\alpha.
\end{equation}

\end{proof}

\begin{proposition}  \label{groundhog.12}
If $0 < \kappa < 8,$ there exists $c < \infty$ such that  the
following holds. 
Suppose $\eta$ is a crosscut of $\Disk$
and   $\gamma$ is   a radial $SLE_\kappa$ curve
from $1$ to $0$ in $\Disk$. Then
\[  \Prob\{\gamma(0,\infty ) \cap \eta \neq \emptyset \}
  \leq c \, 
          \left[\frac{\diam (\eta)}{\dist(1,\eta)}\right]^{
\alpha}.\]
\end{proposition}

\begin{proof}     We may assume that $\eta \cap \Disk_1
= \emptyset.$ 
By Lemma \ref{deterlemma},  for $n \geq 5$,
conditioned on 
$\gamma[0,\rho_n] = \emptyset,$ we know that
\[       L_{\rho_n} \leq c \, L[\eta] \, e^{-n/2}
               L_{\rho_n}^* . \]
Since $\rho_{n+1} - \rho_n$ is uniformly bounded in
$n$, we can use Lemma \ref{groundhog.1} to conclude
that 
\[ \Prob\{ \gamma[0,\rho_5] \cap \eta
 \neq \emptyset \} \leq  c \, \left[\frac{L[\eta]}
  {L^*[\eta]} \right]^\alpha , \]
and for $n \geq 5$, 
\[    \Prob\{\gamma[0,\rho_{n+1}) \cap \eta \neq
\emptyset \mid \gamma[0,\rho_n] \cap \eta = \emptyset \}
   \leq c\, L[\eta]^\alpha \, e^{-n\alpha/2}
   \leq c \, e^{-n\alpha/2} \, \left[\frac{L[\eta]}
  {L^*[\eta]} \right]^\alpha . \]
By summing over $n$ we get the proposition.

\end{proof}

\begin{proof} [Proof of Theorem \ref{radmain}]
We start by proving the stronger result for $\kappa \leq 4$. 
 Note that
$\p \Disk_{k} \cap H_n $ is a disjoint union   of crosscuts 
$ \eta =\{e^{-k + i\theta}: \theta_{1,\eta} < \theta < \theta_{2,\eta}\}$.
  For each
$\eta$, we use Lemma \ref{deterlemma} and Proposition
\ref{groundhog.12} to see that
\[
   \Prob \left  
\{\gamma[\rho_{n+k},\infty) \cap
\eta \neq \emptyset  \mid \F_{\rho_n}
 \right\} \leq c \, e^{-n \alpha/2} \, (\theta_2-\theta_1)^
              {\alpha}. 
\]
However, since $\alpha \geq 1$ (here we use the
fact that $\kappa \leq 4$),
\begin{equation}  \label{may18.3}
  \sum_{\eta} (\theta_{2,\eta} - \theta_{1,\eta})^{\alpha}
  \leq \left[\sum_{\eta} (\theta_{2,\eta} - \theta_{1,\eta})\right]^\alpha
    \leq (2\pi  )^\alpha.
\end{equation}
 
 We will now prove \eqref{apr4.100} assuming only $\kappa < 8$.
  Let $E = E_{j,k,n}$ denote the
event   $\gamma[\rho_k,\rho_{n+k}] \subset \Disk_j$.
 Lemma \ref{jul19.lemma1} implies that
on the event $E$, there is a unique
crosscut $l \in {\mathcal A}_{n+k,k}$ such that every curve from
the origin to $\p \Disk_j$ in $H_{n+k}$ intersects $l$.  Hence,
on $E$
\[  \Prob\left\{\gamma[\rho_{n+k},\infty)  \not\subset \Disk_j
        \mid    \G_{{n+k}}  \right\} \]
is bounded above by the supremum of
\[  \Prob\left \{\gamma[\rho_{n+k}, \infty) \cap l \neq \emptyset
        \mid    \G_{{n+k}} \right \} , \]
where the supremum is over all $l \in {\mathcal A}_{n+k,k}$.
  For each
such crosscut $l$, we use Lemma \ref{deterlemma} and Proposition
\ref{groundhog.1} to see that  
\[
   \Prob \left  
\{\gamma[\rho_{n+k},\infty) \cap
l \neq \emptyset  \mid \G_{n}
 \right\} \leq c \, e^{-n \alpha/2}
              . 
\]
%
\end{proof}

\subsection{Two-sided radial $SLE_\kappa$}  \label{tworadsec}

In order to prove that two-sided radial $SLE_\kappa$
is continuous at the origin, we will prove the following
estimate.  It is the analogue of Proposition \ref{groundhog.12}
restricted to the crosscut that separates the origin from
$-1$.

\begin{proposition}  \label{may10.prop1}
 If $\kappa < 8$ 
there exist $c'$ such if $\gamma$
is two-sided radial from $1$ to $-1$ through $0$ in
$\Disk$, then for all $k,n > 0$, if $l = l_{n+k.k}^*$,
\begin{equation}  \label{apr4.0}
 \Prob\{\gamma[\rho_{n+k},\infty) \cap \overline
  {l}  \neq \emptyset \mid \G_{n+k}
\} \leq c' \, e^{-n\alpha/2}.
\end{equation}
\end{proposition}

\begin{proof} 
Let $\rho = \rho_{n+k}$ and as
 in   Lemma \ref{may18.lemma1}, let $\psi = \psi_{n+k.k}^*$
be the first time $t \geq \rho$ that $\gamma(t) \in \overline  l$.
It suffices to show that
\[    \Prob\{\psi<  \rho_{n+k+1} 
  \mid \G_{n+k}
\} \leq c \, e^{-n\alpha/2},\]
for then we can iterate and sum over $n$. By
Lemma \ref{may18.lemma1}, we know that 
\begin{equation}  \label{dec9.1}
   S_\psi(0) \leq c \, e^{-n/2} \, S_\rho(0) . 
   \end{equation}
Also,
\eqref{newer} gives $\rho_{n+k+1} -
\rho \leq c_1$ for some uniform $c_1 < \infty$.
Recalling that two-sided $SLE_\kappa$ corresponds to
the radial Bessel equation \eqref{jan30.1} with
$\beta = 2a$, we see from Lemma
\ref{cinco.lemma}, that
\[   \Prob\left\{\min_{\rho \leq t \leq  \rho +c_1}
         S_t(0) \leq \epsilon \, S_\rho(0) \mid \G_{n+k}\right\} \leq   c
               \epsilon^{4a - 1} = c \, \epsilon^{\alpha}.\]
Combining this with \eqref{dec9.1} gives the first inequality.
\end{proof}

\begin{proof}[Proof of Theorem \ref{twomain}]
To prove \eqref{apr4.1}, 
we recall Lemma \ref{jul19.lemma1}
which tells us that if $
\gamma[\rho_{k},\rho_{n+k} ] \subset \Disk_j$, then in order
for $\gamma[\rho_{n+k},\infty)$ to intersect $\Disk_j$ is is
necessary for it to intersect $\overline l$.
\end{proof}

\subsection{Proof of Theorem \ref{maincorollary}} 

Here we finish the proof of Theorem \ref{maincorollary}.
We have already proved the main estimates \eqref{apr4.100}
and \eqref{apr4.1}.  The proof is essentially the same
for radial and two-sided radial; we will do the two-sided
radial case.  We will use the following lemma which has
been used by a number of authors to prove exponential
rates of convergence, see, e.g., \cite{BFG}.  Since it is not very long,
we give the proof. An important thing to note about the
proof is that it does not give a good estimate for
the exponent $u$.

\begin{lemma}  \label{aprillemma}
Let $\epsilon_j$ be a decreasing sequence of
numbers in $[0,1)$ such that
\begin{equation}  \label{apr4.5}
           \limsup_{n \rightarrow \infty} \epsilon_n^{1/n} 
                            < 1. 
   \end{equation}  
 Then there exist $c,u$ such that the following holds.
Let  $X_n$ be  a discrete time Markov chain 
on state space $\{0,1,2\ldots\}$ with transition probabilities
\[   p(j,0) = 1-p(j,j+1)  \leq  \epsilon_j.\]
 Then, 
 \[            \Prob\{X_n < n/2 \mid  X _0 = 0 \} \leq c \, e^{-nu} . \]
\end{lemma}

 \labove \textsf
{\Heuristic  \begin{small}  
 The assumption
that $\epsilon_j$ decrease is not needed since one can
always consider $\delta_j = \min\{\epsilon_1,\ldots,
\epsilon_j\}$ but it makes the coupling argument
described below easier.
 \end{small}} 
\lbelow

\begin{proof}  We will assume that $p(j,0) = \epsilon_j$.
The more general result can be obtained by a simple
coupling argument defining $(Y_n,X_n)$ on the same
space where 
\[  \Prob\{Y_{j+1} =0 \mid Y_j = n \}
  = 1 - \Prob\{Y_{j+1} = n+1 \mid Y_j = n \} = \epsilon_n,\]
in a way such that $Y_n \leq X_n$ for all $n$.

 Let $p_n = \Prob\{X_n = 0\mid X_0 = 0\}$, with
corresponding generating function
\[                  G(\xi) = \sum_{n=0}^\infty p_n \, \xi^n. \]
Let
\[ \delta =  \Prob\left\{X_n \neq 0 \mbox{ for all } n \geq 1 \mid
X_0 = 0\right\} = \prod_{n=0}^\infty [1-p(n,n+1)] 
 = \prod_{n=0}^\infty [1-\epsilon_n] > 0  . \]
For $n \geq 1$, let 
\[  \Prob\{X_n = 0; X_j \neq 0 , 1\leq j \leq n-1 \mid X_0 = 0\}\]
with generating function
\[   F(\xi) = \sum_{n=1}^\infty q_n \, \xi^n . \]
Note that
\[    q_n =  p(0,1) \, p(1,2) \, \cdots \, p(n-2,n-1) \, p(n-1,0) \leq  
  \epsilon_{n-1}. \]
  Therefore,  \eqref{apr4.5} implies that
  the radius of convergence of $F$ is strictly greater than $1$.  Since
  $F(1) = 1 - \delta < 1$, we can find $t > 1$ with $F(t) < 1$, 
  and hence
 \[            G(t) = [1-F(t)]^{-1} < \infty , \]
 In particular,
 if $e^{2u} < t$, then there exists $c< \infty$ such that for all $n$,
 \[                        p_n \leq c  \, e^{-2un}. \]
 Let $A_n$ be the event that $X_m = 0$ for some $m \geq n/2$.  Then,
 \[        \Prob(A_n) \leq \sum_{j \geq n/2} p_j \leq c' \, e^{-un}. \]
 But on the complement of $A_n$, we can see that $X_n \geq n/2$.

 \end{proof}

\begin{proof}[Proof of Theorem \ref{maincorollary}]
The proof is the same for radial or two-sided $SLE_\kappa$.
Let us assume the latter.  
  The important observation is that
for every $0 < k < m< \infty$, we can find $\epsilon > 0$ such that for all $n$,
\[\Prob\{\gamma[\rho_{n+5+k},\rho_{n+m+5+k}] \subset \Disk_{n+5}  \mid 
\G_n\} \geq \epsilon.\]
(This can be shown by considering the event that the driving function stays
almost constant for a long interval of time after $\rho_{n}.$. We omit
the details.)  By combining this with Proposition
\ref{may10.prop1}, we can see that there exists $m,\epsilon$ such that
\begin{equation}  \label{dec9.2}
  \Prob \{\gamma[\rho_{n+m+5},\infty) \subset \Disk_{n+5}  \mid 
\G_n\} \geq \epsilon.
\end{equation}

To finish the argument, let us fix $k$. Let $c', u = \alpha/2$
be the constants from \eqref{apr4.1} and let $m$ be sufficiently
large so that $c' e^{-nu} \leq 1/2$ for $n \geq m$.
For positive integer $n$ define
$L_n$ to be the largest integer $j$ such that
\[                         \gamma[\rho_{n+k-j}, \rho_{n+k}] \subset  \Disk_k . \]
The integer $j$ exists but could equal zero.  From \eqref{apr4.1},
we know that 
\[             \Prob\left\{L_{n+k+1}  = L_{n+k}+ 1 \mid \G_{n+k} \right\}
                \geq   1 - c' \, e^{-nL_{n+k}u},\]
 and if $L_{n+k} \geq m$, the right-hand side is greater than $1/2$.
 
We see that the distribution of $L_{n+k}$ is stochastically bounded below
by that of a Markov chain $X_n$ of the type in Lemma \ref{aprillemma}.
     Using this we see that there exists $C',\delta$ such that 
\[   \Prob\{L_{n+k} \leq n/2 \mid \G_k \} \leq  C' \, e^{-\delta n} . \] 
On the event $\Prob\{L_{n+k} \geq n/2\}$, we can use 
\eqref{apr4.0} to conclude that the conditional probability of returning
to $\p \Disk_k$ after time $\rho_{n+k}$ given $L_{n+k} \geq n/2$
is $O(e^{-n\alpha/4})$.  This completes the proof with $u = \min\{\delta,\alpha/4
\}$.

\end{proof}

\bibliography{apr11}{}

\begin{thebibliography}{1}

\bibitem{AK}
Tom Alberts and Michael~J. Kozdron.
\newblock Intersection probabilities for a chordal {SLE} path and a semicircle.
\newblock {\em Electron. Commun. Probab.}, 13:448--460, 2008.

\bibitem{BFG}
Xavier Bressaud, Roberto Fern{\'a}ndez, and Antonio Galves.
\newblock Decay of correlations for non-{H}\"olderian dynamics. {A} coupling
  approach.
\newblock {\em Electron. J. Probab.}, 4:no. 3, 19 pp. (electronic), 1999.

\bibitem{LPark}
G.~Lawler.
\newblock Schramm-{L}oewner evolution ({SLE}).
\newblock In {\em Statistical mechanics}, volume~16 of {\em IAS/Park City Math.
  Ser.}, pages 231--295. Amer. Math. Soc., Providence, RI, 2009.

\bibitem{LBrent}
G.~Lawler and B.~Werness.
\newblock {Multi-point Green's function and an estimate of Beffara}.
\newblock preprint.

\bibitem{LBook}
G.F. Lawler.
\newblock {\em {Conformally invariant processes in the plane}}.
\newblock Amer Mathematical Society, 2008.

\bibitem{LSWrest}
Gregory Lawler, Oded Schramm, and Wendelin Werner.
\newblock Conformal restriction: the chordal case.
\newblock {\em J. Amer. Math. Soc.}, 16(4):917--955 (electronic), 2003.

\bibitem{LLind}
Gregory~F. Lawler and Joan~R. Lind.
\newblock Two-sided {${\rm SLE}_{8/3}$} and the infinite self-avoiding polygon.
\newblock In {\em Universality and renormalization}, volume~50 of {\em Fields
  Inst. Commun.}, pages 249--280. Amer. Math. Soc., Providence, RI, 2007.

\bibitem{Schramm}
O.~Schramm.
\newblock {Scaling limits of loop-erased random walks and uniform spanning
  trees}.
\newblock {\em Israel Journal of Mathematics}, 118(1):221--288, 2000.

\bibitem{RS}
O.~Schramm and S.~Rohde.
\newblock {Basic properties of SLE}.
\newblock {\em Annals of mathematics}, 161(2):883, 2005.

\end{thebibliography}
\bibliographystyle{plain}

\end{document}